\documentclass[reqno,12pt]{amsart}
\usepackage{txfonts}
\usepackage{mathrsfs}
\usepackage{amsmath}
\usepackage{amssymb}

\usepackage{amsmath,amssymb,amsthm}
\usepackage{verbatim}
\usepackage{latexsym, bm}

\title[On the Griffiths numbers]{On the Griffiths numbers for higher dimensional
singularities\\
Sur les nombres de Griffiths pour les singularit$\acute{\text{E}}$s
de dimension sup$\acute{\text{E}}$rieure}

\author[Rong Du]{Rong Du$^{\dag}$}

\address{Department of Mathematics\\
Shanghai Key Laboratory of PMMP\\
East China Normal University\\
Rm. 312, Math. Bldg, No. 500, Dongchuan Road\\
Shanghai, 200241, P. R. China} \email{rdu@math.ecnu.edu.cn}
\address{Current address: Department of Mathematics\\
The University of Hong Kong\\
Rm. 408, Run Run Shaw Bldg, Pokfulam, Hong Kong}
\email{rongdu@hku.hk}

\author[Yun Gao]{Yun Gao$^{\dag\dag}$}

\address{Department of Mathematics, Shanghai Jiao Tong University,
Shanghai 200240, P. R. of China}

\email{gaoyunmath@sjtu.edu.cn}

\thanks{$^{\dag}$ The Research Sponsored by National Natural Science Foundation of China, Shanghai Pujiang Program and Scientific Research Foundation for the Returned Overseas Chinese Scholars, State Education Ministry.}
\thanks{$^{\dag\dag}$ The Research Sponsored by National Natural Science Foundation of China and Innovation Program of Shanghai
Municipal Education Commission.}

\theoremstyle{definition}
\newtheorem{theorem}[subsection]{Theorem}
\newtheorem{lemma}[subsection]{Lemma}
\newtheorem{definition}[subsection]{Definition}

\newtheorem{remark}[subsection]{Remark}
\newtheorem{conjecture}[subsection]{Conjecture}
\let\tilde=\widetilde
\allowdisplaybreaks

\def\dashfill{\leaders\hbox{\hbox to 3.25pt{\hrulefill}\hspace*{2pt}\hbox to 3.25pt{\hrulefill}}\hfill}
\newcommand{\CITE}[1]{{[#1]}}
\let\cite=\CITE

\begin{document}

\maketitle
\begin{abstract}
We show that Yau's conjecture on the inequalities for (n-1)-th
Griffiths number and  (n-1)-th Hironaka number does not hold for
isolated rigid Gorenstein singularities of dimension greater than 2.
But his conjecture on the inequality for (n-1)-th Griffiths number
is true for irregular singularities.\\\\

\noindent R$\acute{\text{e}}$sum$\acute{\text{e}}$. \ Nous montrons
que la conjecture de Yau sur
l'in$\acute{\text{e}}$galit$\acute{\text{e}}$ du
(n-1)-i$\grave{\text{e}}$me nombre de Griffiths et celle du
(n-1)-i$\grave{\text{e}}$me nombre de Hironaka sont non-valables en
g$\acute{\text{e}}$n$\acute{\text{e}}$ral pour les
singularit$\acute{\text{e}}$s de Gorenstein isol$\acute{\text{e}}$es
rigides de dimension sup$\acute{\text{e}}$rieure $\grave{\text{a}}$
2. Cependant, la premi$\grave{\text{e}}$re conjecture sur
l'in$\acute{\text{e}}$galit$\acute{\text{e}}$ du
(n-1)-i$\grave{\text{e}}$me nombre de Griffiths est vraie pour les
singularit$\acute{\text{e}}$s
irr$\acute{\text{e}}$guli$\grave{\text{e}}$res.
\end{abstract}
\vspace{.5cm}


\vspace{.3cm} {\small{Mots-cl$\acute{\text{e}}$s:} nombre de
Griffiths, nombre de Hironaka, singularit$\acute{\text{e}}$s de
Gorenstein rigides, singularit$\acute{\text{e}}$s
irr$\acute{\text{e}}$guli$\grave{\text{e}}$res.}

\vspace{.5cm} {\small{AMS subject classifications: 32S05, 14B05.}}

\section{\textbf{Introduction}}
In singularity theory, one always wants to find invariants
associated to singularities. Let $(V, o)$ be a Stein analytic space
with $o$ as its only singularity of dimension $n\ge 2$. In
\cite{Ya1}, Yau introduced a bunch of invariants which are naturally
attached to isolated singularities. These invariants are used to
characterize the different notions of sheaves of germs of
holomorphic differential forms on analytic spaces. Various formulas
which relate to all these invariants were proved in \cite{Ya1}.
Among these invariants the Griffiths number $g^{(p)}$, the Hironaka
number $h^{(p)}$ and $\delta^{(p)}$ are the most interesting
invariants. In 1981, Yau conjectured that the following two
inequalities of these invariants should be true for general isolated
normal singularities in \cite{Ya2}.

\begin{conjecture}\label{yau}
Let $(V, o)$ be a Stein analytic space with $o$ as its only normal
singularity of dimension $n\ge 2$. Then
\begin{enumerate}
\item [(i)] $g^{(n-1)}\ge n-1$;
\item [(ii)] $\delta^{(n-1)}\ge h^{(n-1)}+n-1$.
\end{enumerate}
\end{conjecture}

In (\cite{Ya2}), he confirmed his conjecture for surface
singularities and non-rational singularities with ``good"
$\mathbb{C}^*$-action of dimension greater than 2. As an application
, he showed that any Gorenstein surface singularities with ``good"
$\mathbb{C}^*$-action are not rigid.  Although this conjecture
inspired researches in singularity theory for long time, we will
show that this conjecture is not true in general. But the first part
of his conjecture holds for irregular singularities.

\section{\textbf{Preliminaries}}
Let $(V,o)$ be a normal isolated singularity of dimension $n\ge 2$.
It is well known that holomorphic functions defined on $V-\{o\}$ can
be extended across $o$.  However for holomorphic forms, the
situation is completely different. Even if we assume that the
holomorphic forms defined on $V-\{o\}$ are $L^2$-integrable in a
neighborhood of $o$ in the sense of Griffiths (\cite{Gr}), it is not
clear whether holomorphic forms can be extended across $o$. In
\cite{Ya1}, the Griffiths number $g^{(p)}$ was introduced to measure
how many $L^2$-integrable holomorphic $p$-forms on $V-\{o\}$ cannot
be extended across $o$. Similarly, Yau defined another class of
invariants $\delta^{(p)}$ which measures how many holomorphic
$p$-forms on $V-\{o\}$ cannot be extended across $o$ in \cite{Ya2}.

In \cite{Ya1}, Yau studied the relations among all kinds of sheaves
of germs of holomorphic forms which were also considered by
Grauert-Grothendieck, Noether, Ferrari and Siu.

\begin{enumerate}
\item
Noether: $\bar{\Omega}^p_V:=\pi_*\Omega^p_M$, where $\pi:
M\longrightarrow V$ is a resolution of singularities of $V$.
\item
Grauert-Grothendieck:
$\Omega^p_V:=\Omega_{\mathbb{C}^N}^p/\mathscr{K}^p$, where
$\mathscr{K}^p=\{f\alpha+dg\wedge\beta :
\alpha\in\Omega_{\mathbb{C}^N}^p; \beta\in
\Omega_{\mathbb{C}^N}^{p-1}; f, g\in\mathscr{I}\}$ and $\mathscr{I}$
is the ideal sheaf of $V$ in $\mathbb{C}^N$.
\item
Ferrari:
$\widetilde{\Omega}^p_V:=\Omega_{\mathbb{C}^N}^p/\mathscr{H}^p$,
where $\mathscr{H}^p=\{\omega\in\Omega_{\mathbb{C}^N}^p:
\omega|_{V\backslash V_{sing}}=0\}$.
\item
Siu: $\bar{\bar{\Omega}}^p_V:=\theta_*\Omega^p_{V\backslash
V_{sing}}$ where $\theta : V\backslash V_{sing}\longrightarrow V$ is
the inclusion map and $V_{sing}$ is the singular set of $V$.
\end{enumerate}

If $V$ is a normal variety, then dualizing sheaf $\omega_V$ of
Grothendieck is actually the sheaf $\bar{\bar{\Omega}}^n_V$. Clearly
${\Omega}^p_V$, $\widetilde{\Omega}^p_V$ are both coherent.
$\bar{\Omega}^p_V$ is a coherent sheaf because $\pi$ is a proper
map. $\bar{\bar{\Omega}}^p_V$ is also a coherent sheaf by the
theorem of Siu (see Theorem A of \cite{Si}).

\begin{lemma}\label{exact}
There are several short exact sequences:
\begin{equation}\label{K}
0\longrightarrow K^p\longrightarrow\Omega_V^p\longrightarrow
\tilde{\Omega}_V^p\longrightarrow 0,
\end{equation}
\begin{equation}\label{H}
0\longrightarrow\tilde{\Omega}_V^{p}\longrightarrow\bar{\Omega}_V^{p}\longrightarrow
H^{p}\longrightarrow 0,
\end{equation}
and
\begin{equation}\label{J}
0\longrightarrow\bar{\Omega}_V^{p}\longrightarrow\bar{\bar{\Omega}}_V^{p}\longrightarrow
J^{p}\longrightarrow 0,
\end{equation}
where $K^p$, $H^p$ and $J^p$ are all supported on the singular set.
\end{lemma}

\begin{definition}\label{mgs}
Let $(V, o)$ be a Stein analytic space with $o$ as its only
singularity. Let $K^p$, $H^p$ and $J^p$ be defined as in (\ref{K}),
(\ref{H}), and (\ref{J}). Then the invariants $m^{(p)}$, $g^{(p)}$
and $s^{(p)}$ at $o$ are defined to be dim$K^p_o$, dim$H^p_o$ and
dim$J^p_o$ respectively. $\delta^{(p)}$ is defined to be the
dimension of the co-kernel of the natural map
${\Omega}_V^{P}\rightarrow\bar{\bar{\Omega}}_V^{p}$. We also define
the geometric genus of the singularity $p_g$ to be $s^{(n)}$ and
irregularity of the singularity $q$ to be $s^{(n-1)}$.
\end{definition}

The following lemma can be found as Lemma 2.7 in \cite{Ya1}. We will
provide a short proof here.
\begin{lemma}\label{gp}
Let $(V, o)$ be a Stein analytic space with $o$ as its only isolated
singularity. Let $\pi: M\rightarrow V$ be a resolution of the
singularity. Then
\[g^{(p)}=\text{dim}\Gamma(M, \Omega_M^p)/\pi^*\Gamma(V, \Omega_V^p),\]
and
\[\delta^{(p)}=\text{dim}\Gamma(V-\{o\}, \Omega_V^p)/\Gamma(V, \Omega_V^p),\]
where we identify $\Gamma(V, \Omega_V^p)$ with its image in
$\Gamma(V-\{o\}, \Omega_V^p)$.
\end{lemma}
\begin{proof}
Because $K^p$ is supported on the isolated singularity $o$ and
$\tilde{\Omega}_V^p$ is coherent on the Stein space $V$, they both
have vanishing cohomology groups of degree greater than $0$. So,
from (\ref{K}) and (\ref{H}), we can have an exact sequence as
follows
\begin{equation*}
0\longrightarrow
K^p_o\longrightarrow\Gamma(V,\Omega_V^p)\longrightarrow \Gamma(V,
\bar{\Omega}_V^{p})\longrightarrow H^{p}_o\longrightarrow 0.
\end{equation*}
Therefore we get the first equality by identifying $\Gamma(V,
\bar{\Omega}_V^{p})$ with $\Gamma(M, \Omega_M^p)$. The second
equality is obvious by the definition.
\end{proof}

\section{\textbf{Inequalities for invariants of singularities}}
We will show that Yau's conjecture mentioned in the first section is
not true in general.
\begin{theorem}\label{counter ex} Let $(V, o)$ be an isolated rigid Gorenstein
singularity of dimension $n\ge 3$. Then $(V, o)$ does not satisfy
the above two inequalities.
\end{theorem}
\begin{proof}
From the exact sequences in Lemma \ref{exact} one immediately
concludes the inclusions of finite dimensional vector spaces
\[\tilde{\Omega}_V^{p}\subset\bar{\Omega}_V^{p}\subset\bar{\bar{\Omega}}_V^{p}\]
from which the equality $\delta^{(p)}=g^{(p)}+s^{(p)}$ immediately
follows (see Lemma 1.3 in \cite{Ya2}).

A rigid Gorenstein singularity has $\delta^{(n-1)}=0$ (cf. Theorem
4.6 in \cite{Ya2}) from which it follows that $g^{(n-1)}$. Hence a
rigid Gorenstein singularity provides a counterexample to the
conjectures that $g^{(n-1)}\ge n-1$ and $\delta^{(n-1)}\ge p_g+n-1$.

%On the other hand, quotient singularities are all rational (see
%\cite{Vi} Proposition 1), i.e. the geometry genus $p_g=h^{(n-1)}$ of
%$(V, o)$ is $0$. From the result of Van Straten and Steenbrink (see
%\cite{St-St} Corollary 1.4), the irregularity $q$ is not greater
%than $p_g$. So $q=s^{(n-1)}=0$. Therefore
%$\delta^{(n-1)}=g^{(n-1)}=0$ which contradicts to the statement of
%the conjecture.
\end{proof}

\begin{remark}
The isolated rigid Gorenstein singularities of dimension $n\ge 3$ do
exist. For example, the existence of an isolated Gorenstein finite
quotient singularity of dimension $n\ge 3$ is given by \cite{M-B-D},
\cite{Ya-Yu} or \cite{An}. From Theorem 3 in \cite{Sc}, such
singularity is rigid.
\end{remark}

Next, we are going to show that the inequality (\ref{i}) holds for
irregular singularities.

\begin{definition}
A normal isolated  singularity is called regular if the irregularity
$q=0$. Otherwise, it is called irregular.
\end{definition}

\begin{theorem}
Let $(V, o)$ be a normal isolated irregular singularity of dimension
$n\ge 2$. Then
\begin{equation*}
g^{(n-1)}\ge n-1.\label{i}
\end{equation*}
\end{theorem}
\begin{proof}
We will improve the methods used in Yau's paper (\cite{Ya2} Theorem
3.1) to estimate the lower bound of $g^{(n-1)}$. Let $\pi: (M,
A)\rightarrow (V, 0)$ be a resolution of the singularity with
$A=\cup_{i=1}^{s} A_i$ as the exceptional set, where each $A_i$ is
an irreducible component. We can assume without loss of generality
that the exceptional set $A$ is a divisor in $M$ with normal
crossings. Embed $(V, o)$ into $(\mathbb{C}^m, 0)$ whose coordinate
functions are supposed to be $z_1,\cdots, z_m$. Since the
singularity is irregular, we know that
$$q=\text{dim} \Gamma(M\backslash A, \Omega_M^{n-1})/\Gamma(M,
\Omega_M^{n-1})>0.$$ So there exists a holomorphic $(n-1)$-form
$\omega_0$ on $M\backslash A$ that can not be extended across $A$,
i.e. $\omega_0\in \Gamma(M\backslash A, \Omega^{n-1}_M)\backslash
\Gamma(M, \Omega^{n-1}_M)$, which means that $\omega_0$ must have
poles along some irreducible component $A_1$ of $A$. Suppose
$\omega_1$ has the lowest order of poles along $A_1$ among all
$\omega\in \Gamma(M\backslash A, \Omega^{n-1}_M)\backslash \Gamma(M,
\Omega^{n-1}_M)$. Then $\pi^*(z_j)\omega_1$ is holomorphic along
$A_1$ for all $j$, $1\le j\le m$. Otherwise, it contracts to the
assumption. If $\pi^*(z_j)\omega_1\notin \Gamma(M, \Omega^{n-1}_M)$
for some $j$, it must have poles along another irreducible component
$A_2$ of $A$. Suppose $\omega_2\in\Gamma(M\backslash A,
\Omega^{n-1}_M)$ is holomorphic along $A_1$ and has the lowest order
of poles along $A_2$. Then $\pi^*(z_j)\omega_2$ must be holomorphic
along $A_1$ and $A_2$, for all $j$, $1\le j\le m$. Similarly, we can
continue such kind of steps to generate $\omega_3$, $\omega_4$,
$\cdots$ if it is available. Since the number of irreducible
components of $A$ is finite, by induction, there exists a non-empty
set $W_k=\{\omega\in \Gamma(M\backslash A, \Omega^{n-1}_M)\backslash
\Gamma(M, \Omega^{n-1}_M): \omega$ has poles along some irreducible
component $A_k$ and holomorphic along $A_1,\dots, A_{k-1}$ such that
$\pi^*(z_j)\omega\in \Gamma(M, \Omega^{n-1}_M)$ for all $j$, $1\le
j\le m$$\}$.

Suppose $\omega_k\in W_k$. We will separate our argument into two
parts according to the order of poles of $\omega_k$.

\begin{itemize}
\item
The order of poles of $\omega_k$ is greater than $1$ along some
exceptional component $A_k$:

Choose a point $b$ in $A_k$ which is a smooth point of $A$. Let
$(x_1, x_2,$ $\cdots, x_n)$ be a coordinate system centered at $b$
such that $A_k$ is given locally by $x_1=0$ at $b$. Take the power
series expansion of $\pi^*(z_j)$ around $b$:
\begin{equation}\label{z}
\pi^*(z_j)\circeq x_1^{r_j}f_j, 1\le j\le m,
\end{equation}
where $f_j$ is a holomorphic function such that $f_j(0, x_2, \cdots,
x_n)\neq 0$ and ``$\circeq$" means local equality around $b$.
Without loss of generality, we may assume $r_1=min\{r_1, \dots,
r_m\}$. It is easy to see that the holomorphic $(n-1)$-form
\begin{equation}\label{wedge}
d\pi^*(z_{i_1})\wedge d\pi^*(z_{i_2})\wedge \cdots \wedge
d\pi^*(z_{i_{n-1}})
\end{equation}
has vanishing order at least $(n-1)r_1-1$ along $A_k$.  So the
vanishing orders of the $(n-1)$-form $(\pi^*(z_1))^j\omega_k$, $1\le
j\le n-1$, along $A_k$ are at most $(n-1)r_1-2$. These $(n-1)$-forms
cannot be the linear combinations of (\ref{wedge}). Therefore we
have produced at least $n-1$ holomorphic $(n-1)$-forms on $M$ which
are not obtained by pulling back of holomorphic $(n-1)$-forms on
$V$.
\item
The maximal order of poles of $\omega_k$ is equal to $1$:

Because
\[H^0(M, \Omega_M^{n-1})=H^0(M, \Omega_M^{n-1}(log A))\] from
Theorem 1.3 in \cite{St-St}, there exists an exceptional component
$A_k$ such that the local expression of $\omega_k$ around some
smooth point $b'$ on $A$ contains a summand of the form
\[x_1^{-1}g dx_2\wedge dx_3\wedge \cdots \wedge dx_n,\] where $g$ is a holomorphic function such that $g(0,
x_2, \cdots, x_n)\neq 0$ and $(x_1, x_2, \cdots, x_n)$ is a
coordinate system centered at $b'$ such that $A_k$ is given locally
by $x_1=0$ at $b'$. Similarly, if we express $\pi^*(z_j)$, where
$1\le j\le m$, locally as (\ref{z}) and still assume $r_1=min\{r_1,
\dots, r_m\}$, then the coefficient before $dx_2\wedge dx_3\wedge
\cdots \wedge dx_n$ in the expression of (\ref{wedge}) has vanishing
order of $x_1$ at least $(n-1)r_1$. However the coefficient before
$dx_2\wedge dx_3\wedge \cdots \wedge dx_n$ in the expression of the
$(n-1)$-form $(\pi^*(z_1))^j\omega_k$, $1\le j\le n-1$, has
vanishing order of $x_1$ at most $(n-1)r_1-1$. These $(n-1)$-forms
cannot be the linear combinations of (\ref{wedge}). Therefore we
also have produced at least $n-1$ holomorphic $(n-1)$-forms on $M$
which are not obtained by pulling back of holomorphic $(n-1)$-forms
on $V$.
\end{itemize}
\end{proof}
\section*{Acknowledgements}
The first author would like to thank S.S.-T. Yau for useful
discussions. The first author would also like thank N. Mok for
supporting his research when he was in the University of Hong Kong.
Finally, both authors heartily thank the referees for their
invaluable suggestions and comments. In particular, both authors are
grateful to the referee who helped us to shorten the proof of
Theorem \ref{counter ex}.

%\end{comment}

\end{document}